\documentclass[a4paper,12pt]{article}

\usepackage{epsfig}
\usepackage{amsfonts,eucal}
\usepackage{amssymb,amsmath,amsthm,color}
\usepackage[normalem]{ulem}

\usepackage[T2A]{fontenc}
\usepackage[cp1251]{inputenc}
\newtheorem{remark}{Remark}[section]
\newtheorem{theorem}{Theorem}[section]

\newtheorem{lemma}{Lemma}[section]

\newcommand{\Ga}{\Gamma}
\newcommand{\ga}{\gamma}
\newcommand{\BB}{{\mathbb B}}
\newcommand{\XX}{{\mathbb X}}
\newcommand{\R}{{\mathbb R}}
\newcommand{\N}{{\mathbb N}}

\newcommand{\X}{{\R^d}}
\newcommand{\XXX}{{\mathfrak{X}}}
\newcommand{\La}{\Lambda}

\newcommand{\B}{\mathcal{B}}

\newcommand{\M}{\mathcal{M}}

\begin{document}

\title{Persistence in perturbed contact models in continuum}

\author{Sergey Pirogov\thanks{
Institute for Information Transmission Problems, Moscow, Russia
(s.a.pirogov@bk.ru).} \and Elena Zhizhina\thanks{Institute for
Information Transmission Problems, Moscow, Russia (ejj@iitp.ru).} }

\date{}

\maketitle

\begin{abstract}
Can a local disaster lead to extinction? We answer this question in this work.
In the paper \cite{PZ-PPI} we considered contact processes on locally compact metric spaces with state dependent birth and death rates and formulated sufficient conditions on the rates that ensure the existence of invariant measures. One of the crucial conditions in \cite{PZ-PPI} was the critical regime condition, which meant the existence of a balance between birth and death rates in average. In the present work, we reject the criticality condition and suppose that the balance condition is violated. This implies that the evolution of the correlation functions of the contact model under consideration is determined by a nonlocal convolution type operator perturbed by a (negative) potential. We show that local peaks in mortality do not typically lead to extinction. We prove that a family of invariant measures exists even without the criticality condition and these measures can be described using the Feynman-Kac formula.
\\

\end{abstract}

\section{Introduction}

The contact processes on the lattice have been introduced in the pioneer papers of Harris \cite{H}, Holley and Liggett \cite{HL}, see also the monograph of Liggett \cite{Lig1985}.
In recent years, the study of contact processes running in continuous spaces has attracted great interest,
see e.g. \cite{KKP, KKPZ, KS}. This class of processes is a particular case of continuous time birth and death processes, and particle configurations in continuum appeared to be more adapted to modeling evolutions in various biological systems and models of population dynamics.

Taking into account applications of the contact processes as models describing a spread of epidemic diseases or a population growth, one of the main problem under consideration is to determine the parameters of the model providing stationary regime and to prove the existence of invariant measures.  The appearance of limiting invariant states was proved only in the so-called critical regime, see e.g. \cite{PZ-EJP, PZ-PPI}, that is, when there is a certain stochastic  balance between birth and death.
As was shown in \cite{KKP}, there exists a continuum of invariant measures for the contact processes in $\X, \ d\ge 3,$ in the critical regime with a constant death rates. In small dimensions $d=1,2,$ the existence of the stationary regime depends on the behavior of the dispersal kernel at infinity. It was proved in \cite{KKPZ} that for the contact processes in $\X, \ d= 1,2,$ invariant measures exist only if the dispersal kernel has a heavy tail at infinity. In the case of light tails the pair correlation function grows to infinity as $t \to \infty$, and hence invariant measures do not exist. Thus heavy tails of dispersal kernels appear to make the critical regime more stable contrary to light tails.

The existence of invariant measures in the marked contact model in $\X, \ d\ge 3,$ with a compact spin space and for constant death rates was proved in \cite{KPZh}. Such models are used, in particular, to describe evolution in quasi-species populations with mutations, see \cite{Nov}.

A more general framework has been investigated in \cite{PZ-EJP}, where we have formulated conditions providing the existence of a one-parameter family of invariant measures in contact processes on general locally compact separable metric spaces. One of the conditions, the so-called transience condition, was formulated in terms of the associated Markov jump process. It means that any pair of independent trajectories of this jump process run away from each other. The transience condition is not needed for a special class of contact processes - the contact processes with immigration. For such processes a unique invariant measure exists in any dimension, see \cite{CHM}, \cite{PZh-MMJ}.


In the present work, we consider contact processes without the assumption of criticality.
We will show that local violation of criticality due to excess mortality does not destroy invariant measures, but only perturbs them. A similar phenomenon exists in the theory of Schr\"{o}dinger operators, when local positive potentials do not change a spectrum of Laplacian, see e.g. \cite[Chapter 13]{RS4}. 
In our study of the contact model, we need not only to solve the equation for the first correlation function specified by a  Schr\"{o}dinger type operator, but also to explore an infinite chain of hierarchical equations for all correlation functions of the model. Thus, the first step of our research is to construct the first correlation function of the stationary system using the Feynman-Kac formula, and then we find  the solution of the whole chain by the inductive procedure.


The case studied here is unlike to the case from \cite{BPZ1}, \cite{BPZ2}, \cite{KMPZ}, where local negative fluctuations of the mortality can lead to an exponential popultion growth in any compact region.

In proving the main results of this work, we adapt the arguments from \cite{KKP, KPZh, PZ-EJP} developed for contact models in the critical regime to the analysis of contact models in the absence of the critical regime condition.
The remainder of this paper is organised as follows. In Section 2, we introduce the model and formulate assumptions on the model. We formulate the main result in Section 3. In Section 4, we give the proof of the main theorem.

\section{The model}


Let $\XXX$ be a locally compact separable metric space, ${\B}({\XXX})$ be its Borel $\sigma$-algebra,
and $m$ will denote a locally finite  Borel measure on ${\B} (\XXX)$, i.e. $m$ is finite on compact sets.
Denote by $\M(\XXX)$ the space of locally finite Borel measures on $\B(\XXX)$ and by ${\B}_{\mathrm{b}} ({\XXX})$
the system of all compact sets from ${\B}({\XXX})$.

A configuration $\gamma \in \Ga(\XXX)$ on $\XXX$ is a finite or countably infinite locally finite
unordered set of points in $\XXX$, and some of them can be multiple, i.e. repetitions are permitted.
If the measure $m \in \M(\XXX)$ is atomic then  $\Ga(\XXX)$ contains configurations with multiple points. This case is realized on graphs with a counting measure $m$.
For the continuous contact models, when $m$ is non-atomic (see e.g. \cite{KKP, KKPZ, KS}), as the phase space $\Ga$
one can take the set of locally finite configurations in $\XXX$ with distinct elements:
\begin{equation} \label{confspace}
\Ga_c =\Ga_c\bigl(\XXX\bigr) :=\Bigl\{ \ga \subset \XXX \Bigm| \ | \ga\cap\La| < \infty, \ \mathrm{for \ all } \ \La \in {\B}_{\mathrm{b}}
(\XXX)\Bigr\},
\end{equation}
where $|\cdot|$ denotes the number of elements of a set.

We can identify each $\ga\in\Ga$ with an integer-valued measure $\sum_{x\in
\gamma }\delta_x\in \M(\XXX)$, where $\delta_x$ is
the Dirac measure with unit mass, and the sum is taken considering the multiplicity of elements in the configuration $\gamma$.
For any $ \La \in {\B}_{\mathrm{b}}(\XXX)$ we denote by $|\gamma \cap \La|$ the value $\gamma(\La)$ of the measure $\gamma$ on $\Lambda$.

The contact model is a continuous time Markov process on $\Gamma(\XXX)$ which is a
particular case of a general birth-and-death process.
The model is given by a heuristic
generator defined on a proper class of functions  $F:
\Gamma \to \R$ as follows:
\begin{align}
 \begin{aligned}\label{generator}
(L F)(\gamma) & = \sum_{ x \in \gamma} U(x) \left(F(\gamma \backslash x) - F(
\gamma)\right) \\
& +   \int\limits_{\XXX} \sum_{x \in \gamma}  a(y,x) (F(\gamma \,\cup
y ) - F(\gamma)) m(dy).
 \end{aligned}
 \end{align}
Notations $\gamma \backslash x$ and $\gamma \,\cup x$  in \eqref{generator} stand for removing and adding one particle at position $x \in \XXX$. Similarly, $x \in \gamma$ refers to any particle in the configuration $\gamma$.
The first term in \eqref{generator} corresponds to the death of a particle at position $x$: each element $x \in \gamma$ of the configuration $\gamma\in\Ga$ can die with the death rate
$$
U(x) = V(x)+ W(x),
$$ where the critical regime condition (see \eqref{2a} below) is valid  for $V(x) > 0$, and $W(x) \ge 0$ is a non-negative local perturbation of $V(x)$, see condition \eqref{5a} below. The second term of \eqref{generator} describes the birth of a new particle in a neighborhood $dy$ of the point $y$ with the
birth density $A(y,\gamma):=\sum_{x \in \gamma} a(y,x)$. It worth noting that this form of the birth rates suggests that it is not known who is the parent of the new particle, since the birth of a new particle at position $y$ has a cumulative rate $\sum_{x \in \gamma}  a(y,x)$.  Function $a(x,y)$ is also called the dispersal kernel.
\medskip


We formulate now assumptions on the birth and death rates that provide the existence of invariant measures.

1. {\it Measurability } condition. Let   $a:\XXX \times \XXX \to [0, \infty)$, $W:  \XXX \to [0, \infty)$ be non-negative bounded measurable functions, and $V: \XXX \to (0, \infty)$ is a strictly positive bounded measurable function:
  \begin{equation}\label{V}
  0 \ < \  V_{{\rm min}} = \inf\limits_{\XXX} V(x) \ \le \ V(x) \ \le  \ \sup\limits_{\XXX} V(x) = V_{{\rm max}} \ < \ \infty;
\end{equation}
 \begin{equation}\label{W}
   0 \ \le \  W_{{\rm min}} = \inf\limits_{\XXX} W(x) \ \le \ W(x) \ \le  \ \sup\limits_{\XXX} W(x) = W_{{\rm max}} \ < \ \infty.
\end{equation}

2. {\it Regularity} condition. There exists a constant $C>0$, such that
\begin{equation}\label{2a-bis}
\sup\limits_{x \in \XXX} \ \int\limits_{\XXX} a(y,x) \, 
m( dy) \ < \ C;
\end{equation}

3. {\it Critical regime} condition for $V$. There exists a strictly positive bounded measurable function $\Psi(x), \; \Psi(x) \ge p_0>0$ such that
\begin{equation}\label{2a}
 \int\limits_{\XXX} a(x,y)\, \Psi(y) \, m(dy) \ = \ V(x) \, \Psi(x) \quad \mbox{for all } x \in \XXX.
\end{equation}
Critical regime condition is a balance condition between birth and death rates for the unperturbed system if particles are distributed according the measure $\Psi(y) \, m(dy)$.
Define the new measure $\bar{\mathfrak{m}}$ and the new intensity of birth $ b(x,y)$ as follows
\begin{equation}\label{nb}
\bar{\mathfrak{m}}(dy)  = \Psi(y) m(dy), \qquad b(x,y) = \frac{a(x,y)}{\Psi(x)},
\end{equation}
then relation  \eqref{2a} is rewritten as
\begin{equation}\label{cr}
 \int\limits_{\XXX} b(x,y)  \bar{\mathfrak{m}}(dy) \ = \ V(x) \quad \mbox{for all } x \in \XXX.
\end{equation}

4. {\it Transience} condition.
Let us consider the continuous time jump Markov process with generator
\begin{equation}\label{LL}
\mathcal{L} f(x) =  \int\limits_{\XXX} b(x,y) \big( f(y) - f(x) \big)  \bar{\mathfrak{m}}(dy).
\end{equation}
Then we assume that for two independent copies $X(t)$ and $Y(t)$ of this process starting with $X(0)=x$ and $Y(0)=y$ the following condition holds
\begin{equation}\label{2b}
\sup\limits_{x,y} \ \int_{0}^{\infty} \mathbb{E}_{x,y} b(X(t), Y(t)) dt < H
\end{equation}
with a constant $H>0$. Moreover, we assume that the integral in \eqref{2b} converges uniformly in $x$, $y$.

5. {\it W-integrability} condition. The perturbation $W(x)$ is non-negative, bounded and satisfies estimate
\begin{equation}\label{5a}
\sup\limits_{x}  \int_{0}^{\infty} \mathbb{E}_{x} W(X(t)) dt < L
\end{equation}
with some $L>0$, and the integral converges uniformly in $x$. Here as above $X(t)$ is the trajectory of the process with generator \eqref{LL} starting at $x$.
\medskip



\begin{remark}\label{R1}
It worth noting that the form of the generator in \eqref{LL}  as well as the transience condition \eqref{2b} do not depend on the function $W(x)$.
\end{remark}

\begin{remark}\label{R2}
The sufficient condition for \eqref{2b} together with required uniform convergence reads
\begin{equation}\label{suf-cond}
\int\limits_0^{\infty}  \sup\limits_{x, y} \ \mathbb{E}_x b(X(t), y) dt < H
\end{equation}
\end{remark}

\begin{proof} Denote by  $p (x,dy,t)$ the transition function of the Markov jump process with generator \eqref{LL} at time $t$. Then we get
$$
\sup\limits_{x,y} \ \int\limits_{0}^{\infty} \mathbb{E}_{x,y} b(X(t), Y(t)) dt =
\sup\limits_{x,y} \ \int\limits_0^{\infty} \int\limits_{\XXX} \int\limits_{\XXX} b(x', y') p(x, dx',t) p(y, dy',t) dt \le
$$
$$
\sup\limits_{y} \ \int\limits_0^{\infty} \int\limits_{\XXX}  \Big( \sup\limits_{x} \  \int\limits_{\XXX}  b(x', y') p(x, dx', t) \Big)  p(y, dy',t)  dt =
$$
$$
\sup\limits_{y} \ \int\limits_0^{\infty}  \int\limits_{\XXX}  \ \Big( \sup\limits_{x} \ \mathbb{E}_x b(X(t), y')  \Big)  p(y, dy',t) dt \le
$$
$$
\int\limits_0^{\infty} \sup\limits_{y} \ \int\limits_{\XXX}  \Big( \sup\limits_{y'} \sup\limits_{x} \ \mathbb{E}_x b(X(t), y') \Big)  p(y, dy',t) dt =
\int\limits_0^{\infty}  \sup\limits_{x, y'} \ \mathbb{E}_x b(X(t), y') dt.
$$
Therefore, condition \eqref{suf-cond} implies the uniform convergence in \eqref{2b}.
\end{proof}

\begin{remark}\label{R3}
The sufficient condition for \eqref{5a} reads
\begin{equation}\label{R3-suf}
\int\limits_0^{\infty}  \sup\limits_{x} \ \mathbb{E}_x W(X(t)) dt < L.
\end{equation}
\end{remark}

It is easy to see that both estimates \eqref{suf-cond} and \eqref{R3-suf} are valid if $\XXX = \mathbb{R}^d, \ d \ge 3, \ a(x,y) = a(x-y)$ is a bounded probability density, $V(x) \equiv 1, \ \Psi(x) \equiv 1,$ and $W(x) \in L^1(\mathbb{R}^d) \cap L^{\infty}(\mathbb{R}^d)$. The proof of this fact follows from the standard (Polya-type) arguments on the transience of non-degenerate homogeneous random walks in $\mathbb{R}^d, \ d \ge 3$,
see e.g. \cite{LL}, \cite{Petrov}, \cite{PZ-PPI}. 
In the case of small dimensions $d=1,2,$ estimates \eqref{suf-cond} and \eqref{R3-suf} hold if the density $a(x,y) = a(x-y)$ has heavy enough tails, see \cite{KKPZ}.


\section{Time evolution of correlation functions. Main results}

Denote by ${\cal M}_{fm}(\Gamma)$ the set of all probability
measures $\mu$ which have finite local moments of all orders, i.e.
$$
\int_{\Gamma} |\gamma \cap \Lambda |^{n} \ \mu (d \gamma) \ < \ \infty
$$
for  all $\Lambda \in {\cal B}_b(\XXX)$ and $n \in N$.

Together with the configuration space $\Gamma$ we define the space  $\Gamma_0$ of finite configurations, and let $\Gamma^{(n)}_{0,\Lambda} = \{\eta \subset \Lambda: \ |\eta| = n  \}$ be the set of $n$-point configurations in $\Lambda \in {\cal B}_b(\XXX)$.
If a measure $ \mu \in {\cal M}_{fm}(\Gamma)$ is locally absolutely continuous with respect  to the Lebesque-Poisson measure
$$
\lambda_{z, \, \bar{\mathfrak{m}}} \ = \ \sum\limits_{n=0}^{\infty} \frac{z^n}{n!} \, {\bar{\mathfrak{m}}}^{\otimes n}, \quad \mbox{i.e. } \;
\lambda_{z, \, \bar{\mathfrak{m}}}(\Gamma^{(n)}_{0,\Lambda}) \ = \ \frac{z^n \,(\bar{\mathfrak{m}}(\Lambda))^n}{n!},
 \; n=0,1, \ldots,
$$
where $\bar{\mathfrak{m}}(\Lambda)= \int_{\Lambda} \bar{\mathfrak{m}}(dx)$, then there exists the corresponding system of the correlation functions, i.e.
densities of the correlation measure with respect to the Lebesque-Poisson measure.
The terminology originates in statistical mechanics, see, for instance, \cite[Ch. 4]{R}.
Denote by ${\cal M}_{\rm{corr}}(\Gamma)$  the subclass of the class ${\cal M}_{fm}(\Gamma)$ consisting of probability measures on $\Ga$ for which correlation functions exists.

The evolution equation for the system of $n$-point correlation functions corresponding to the continuous contact model in $\XXX$ has the following recurrent forms, see e.g. \cite{KKP, PZ-EJP}:
\begin{equation}\label{59}
\frac{\partial k_{t}^{(n)}}{\partial t} \ = \ \hat L_n^{\ast} k_{t}^{(n)} - W_n k_{t}^{(n)}
+  f_{t}^{(n)}, \quad n\ge 1; \qquad k_{t}^{(0)} \equiv 1,
\end{equation}
where
\begin{align}
 \begin{aligned}\label{korf}
 \hat L^{\ast}_n k_t^{(n)}(x_1, &\ldots, x_n)  =  - \Big( \sum_{i=1}^n  V(x_i) \Big) k_t^{(n)}(x_1,
\ldots, x_n) \\
& + \sum_{i=1}^n \int\limits_{\XXX} b(x_i, y) k_t^{(n)}(x_1, \ldots, x_{i-1}, y,
x_{i+1}, \ldots, x_n) \bar{\mathfrak{m}}(dy),
\end{aligned}
\end{align}
and $W_n$ is the operator of multiplication
\begin{equation}\label{W-1}
W_n k_{t}^{(n)} (x_1, \ldots, x_n) \ = \  \Big( \sum_{i=1}^n  W(x_i) \Big)  k_t^{(n)}(x_1,
\ldots, x_n).
\end{equation}
Denote
\begin{equation}\label{Sn}
S_n =  \hat L_n^{\ast} - W_n.
\end{equation}
Then \eqref{59} is rewritten as
 \begin{equation}\label{59W}
\frac{\partial k_{t}^{(n)}}{\partial t} \ = \ S_n k_{t}^{(n)}
+  f_{t}^{(n)}, \quad n\ge 1.
\end{equation}
Here $f_{t}^{(n)}$ are functions on $\XXX^{n}$ defined for $n \ge 2$ by
\begin{equation}\label{f}
f_{t}^{(n)}(x_1, \ldots, x_n) \ = \ \sum_{i=1}^n  k_{t}^{(n-1)}(x_1,
\ldots,\check{x_i}, \ldots, x_n) \sum_{j\neq i} b(x_i, x_j),
\end{equation}
and  $f_{t}^{(1)} \equiv 0$. The notation $\check{x_i}$ means that this coordinate is excluded.

Note that the operator $S_n = \hat L_n^{\ast} - W_n$ is similar to the $n$-particle Schr\"{o}dinger operator with the potential $W_n$ taken with a minus sign.

Let $\XX_{n} = \BB({\XXX}^n)$ be the Banach space of all measurable real-valued bounded functions on $\XXX^n$ with the $\sup$-norm.
Consider the operator $S_n$ as an operator on the Banach space $\XX_{n}$ for any $n\geq 1$. Then it is a bounded linear operator in $\XX_{n}$, and the arguments based on the variation of parameters formula yields that
\begin{equation}\label{61A}
k_{t}^{(n)} \ = \ e^{t S_n} k_{0}^{(n)} \ + \  \int\limits_0^t e^{(t-s) S_n} f_s^{(n)} \ ds,
\end{equation}
where $f_s^{(n)}$ is expressed through $k_s^{(n-1)}$ by (\ref{f}).
Thus, the solution to the Cauchy problem \eqref{59} in  $\XX_{n}$ with arbitrary initial values $k_{0}^{(n)}\in\XX_{n}$ exists and is unique provided $f_{t}^{(n)}$ is constructed recurrently via the
solution to the same Cauchy problem \eqref{59} for $n-1$.

The goal of this paper is to prove the existence of a family of invariant measures for the contact processes that are out of critical regime.
These measures are described in terms of the corresponding correlation functions
$\{k^{(n)}\}_{n\geq 0}$ as solutions to the following system: 
\begin{equation}\label{Last}
S_n k^{(n)} + f^{(n)}=0, \quad n \ge 1, \quad
k^{(0)}\equiv 1,
\end{equation}
where $S_n, \, f^{(n)}$ were defined by \eqref{korf}-\eqref{f}.
In the sequel, we say that $k:\Ga_{0}\to \R$ solves the system \eqref{Last} in the Banach spaces $(\XX_{n})_{n\geq 1}$ if the corresponding $k^{(n)}\in\XX_{n}$, $n\geq 1$ solve \eqref{Last}.

The main result of the paper is the following theorem.

\begin{theorem}\label{main} {\it Assume that all the above conditions \eqref{V} - \eqref{5a} are fulfilled.
Then the following assertions hold.

{(i)} There exists a one-parameter set of probability measures $\mu^{\varrho} \in {\cal M}_{\rm{corr}}(\Gamma)$ on $\Ga$  depending on the parameter $\varrho >0$ such that the
correlation functions $k_{\varrho}: \Ga_{0}\to\R_{+}$ solve (\ref{Last})  in the Banach spaces $(\XX_{n})_{n\geq 1}$.
Moreover, the following bounds hold for all $(x_1, \ldots, x_n)\in {\XXX}^{n}$
\begin{equation}\label{estimate}
k^{(n)}_\varrho (x_1, \ldots, x_n) \ \le   D  {H}^n (n!)^2 \qquad \text{with } \; D = \sum\limits_{n=1}^{\infty} \frac{(\varrho/ H)^n}{(n!)^2}
\end{equation}
where $H$ is the same constant as in  (\ref{2b}).

{(ii)} Let $\{k_{ t}^{(n)}\}_{n\geq 1}$ be the
solution to the Cauchy problem (\ref{59}) with
 initial data $k_0 = \{k_0^{(n)} \}$ corresponding to the Poisson measure $\pi_\varrho$ with intensity $\varrho$:
\begin{equation}\label{k0}
k_0^{(0)}= 1, \quad k_0^{(n)}(x_1, \ldots, x_n) = \varrho^n,  \; n\ge 1.
\end{equation}
Then
\begin{equation}\label{Th1-2}
\| k_{t}^{(n)} \ - \ k_\varrho^{(n)} \|_{\XX_n} \ \to \ 0, \quad t \to \infty, \quad \forall n\geq 1.
\end{equation}
}
\end{theorem}

\medskip
The main strategy of the proof follows the same line as in \cite{PZ-EJP}. However, in the present paper we should modify the previous proof, considering the operator $S_n$ for any $n$ as a sum of the "unperturbed" generator $\hat L_n^{\ast}$ corresponding to the contact process in the critical regime and the "perturbation" $W_n$.

Theorem \ref{main} states that even in the "perturbed case", when $U(x) = V(x)+W(x)$, the invariant measures exist, and consequently local positive fluctuations $W(x)$ of mortality (with respect to the level of mortality $V(x)$ corresponding to the critical regime) does not lead to the extinction.

\section{The proof of Theorem \ref{main}}


The proof of the first statement of Theorem \ref{main} we start with construction of the first correlation function $k^{(1)}$ and then using the
induction in $n\in\N$ we obtain the solution $\{k^{(n)}\}_{n\geq 1}$ of the full system (\ref{Last}) recurrently.

For the first correlation function $k^{(1)}$ we get from \eqref{korf}, \eqref{Sn}  and (\ref{Last}) the following equation
\begin{equation}\label{8}
-U(x) k^{(1)}(x) + \int\limits_{\XXX} b(x,y) k^{(1)}(y) \bar{\mathfrak{m}}(dy) = 0,
\end{equation}
that can be written using the critical regime condition \eqref{cr}  as
\begin{equation}\label{8-bis}
\int\limits_{\XXX} b(x,y) \big( k^{(1)}(y)- k^{(1)}(x) \big) \bar{\mathfrak{m}}(dy) - W(x)k^{(1)}(x) = 0.
\end{equation}
To find a solution of \eqref{8-bis} we first study the corresponding evolution problem:
\begin{equation}\label{8-evol}
\frac{\partial \, u}{ \partial \, t} = {\cal L} u - W u, \qquad u(x,0)= \varrho, \; \; \varrho>0,
\end{equation}
where $ {\cal L}$ is the generator \eqref{LL} of the Markov jump process in $\XXX$. Equation \eqref{8-evol} is a variant of a non-local heat equation with absorption. The solution of \eqref{8-evol} at time $t$ is given by the Feynman-Kac formula
\begin{equation}\label{FK}
u_\varrho (x, t)\ = \ \varrho \, \mathbb{E}_x \big[ \exp \big(- \int\limits_0^t W(X(s)) ds \big) \big],
\end{equation}
see e.g.  \cite[Chapter III]{RW}, \cite[Chapter II]{Simon}. Here $X(t)$ is the trajectory of the process with generator \eqref{LL} starting at $x$.

It worth noting that the solution $u_\varrho (x, t)$ to \eqref{8-evol} is the same as the solution $k_t^{(1)}(x)$ of \eqref{59}.

We define the stationary solution $k^{(1)} = k^{(1)}_\varrho$ as the limit
\begin{equation}\label{FK-limit}
k^{(1)}_\varrho (x) \ = \ \lim_{t \to \infty} u_\varrho(x,t) \ = \ \varrho \, \mathbb{E}_x \big[ \exp \big(- \int\limits_0^\infty W(X(s)) ds \big) \big].
\end{equation}
We use the notation $k^{(1)}_\varrho$ to emphasize the dependence of the stationary solution on the parameter $\varrho>0$.
It follows from condition \eqref{5a} and the Jensen inequality that the limit \eqref{FK-limit} exists and strictly positive:
\begin{equation}\label{Phi-est}
\varrho \, e^{-L} \le  \varrho \, \mathbb{E}_x \big[ \exp \big(- \int\limits_0^\infty W(X(t)) dt \big) \big] \le \varrho.
\end{equation}
We denote this limit by
\begin{equation}\label{Phi}
\Phi_\varrho (x) =   \varrho \, \mathbb{E}_x \big[ \exp \big(- \int\limits_0^\infty W(X(t)) dt \big) \big].
\end{equation}
Thus formula \eqref{FK-limit} defines a positive bounded function $k_\varrho^{(1)}(x) = \Phi_\varrho(x)$, and
$$
S_1 \Phi_\varrho (x) = (\hat L_1^{\ast} - W_1) \Phi_\varrho (x)  = 0.
$$
Clearly $k_\varrho^{(1)}(x) $ is an element of $\XX_{1}$ and it solves \eqref{8-bis} (and \eqref{8}).
We notice that this function can be interpreted as the spatial density of particles of the stationary distribution,
and formulas \eqref{FK}-\eqref{FK-limit} mean that
$$
k_t^{(1)}(x) = e^{t S_1}  k_0^{(1)}(x) = e^{t S_1} \varrho \ \to \  \Phi_\varrho(x) = k_\varrho^{(1)}(x).
$$

{\bf Example.} If $\XXX = \mathbb{R}^d, \; d \ge 3,$ and  $W \in C_0(\mathbb{R}^d)$, i.e. $W$ has a compact support, then formula \eqref{FK-limit} implies that  $k^{(1)}(x) \to \varrho $ as $|x| \to \infty$.
\medskip

Next we will construct a solution to the system (\ref{Last}) satisfying estimates \eqref{estimate}.
Further we need the following lemma.


%

\begin{lemma} \label{3.2}
The operators $e^{t \hat L_n^{\ast}}$ and $e^{t S_n}$, where $\hat L_n^{\ast}$ and $S_n$ were defined in \eqref{korf} and \eqref{Sn} correspondingly, are
positive, i.e. it maps non-negative functions to non-negative functions, and
\begin{equation}\label{positive}
 e^{t S_n}f \ \le \ e^{t \hat L_n^{\ast}}f
\end{equation}
for all $f \ge 0$.
\end{lemma}

\begin{proof}
For the operator $e^{t \hat L_n^{\ast}}$ see the proof in Lemma 4.1, \cite{PZ-PPI}. We will prove the lemma for $e^{t S_n}$.
The operator
$$
A^i k^{(n)}(x_1, \ldots, x_n) \ := \ \int\limits_{\XXX} b(x_i, y) k^{(n)} (x_1, \ldots, x_{i-1}, y,
x_{i+1}, \ldots, x_n) \bar{\mathfrak{m}}(dy).
$$
is positive and bounded on $\XX_{n}$ for any $1\leq i\leq n$.
Set
\begin{align}
\begin{aligned}\label{21}
S^i k^{(n)}(x_1, \ldots, x_n) \ = &\int\limits_{\XXX} b(x_i, y) k^{(n)}(x_1, \ldots, x_{i-1}, y, x_{i+1}, \ldots,
x_n) \bar{\mathfrak{m}}( dy)\\
& - U(x_i) k^{(n)}(x_1, \ldots, x_n).
\end{aligned}
\end{align}
Using the Trotter formula for the sum $A+B$ of two bounded operators:
$$
e^{t(A+B)} = \lim\limits_{n \to \infty} \Big( e^{\frac{t A}{n}} e^{\frac{t B}{n}} \Big)^n
$$
we conclude that
\begin{equation}\label{20L}
e^{t \, S^{i} }f = \lim\limits_{n \to \infty} \Big( e^{t \frac{ A^{i}}{n}} e^{-t \frac{U}{n}} \Big)^n f \ge  e^{- t \, U_{\rm{max}}} \, e^{t \, A^{i}} f \ge 0
\end{equation}
for any non-negative function $f$. Here $U$ is the operator of multiplication on the positive bounded function $U(x)$, and $ U_{\rm{max}} = \sup\limits_{x \in \XXX}(V(x)+W(x))$.

Representations \eqref{korf}, \eqref{W-1}-\eqref{Sn} yield
$$
e^{t S_n} \ = \ \otimes_{i=1}^n \  e^{t \, S^{i}}.
$$
Then taking into account that
\begin{equation}\label{20L}
\otimes_{i=1}^n \ e^{- t \, U_{\rm{max}}} \, e^{t \, A^{i}}
\end{equation}
is a positive operator, we get the desired conclusion.

Relation $\hat {L_n^{\ast}} = S_n + W_n$ with $W_n \ge 0$ immediately implies estimate \eqref{positive}.
\end{proof}

As follows from (\ref{f}), the function $f^{(n)}$ is the sum of
functions of the form
\begin{equation}\label{32}
f_{i,j} (x_1, \ldots, x_n)  =  k^{(n-1)} (x_1,\ldots,\check{x_i},
\ldots, x_n)  b(x_i, x_j), \quad i\neq j.
\end{equation}

We suppose by induction that
$$
k^{(n-1)} (x_1, \ldots, x_{n-1}) \ \le \ K_{n-1}, \quad \text{for all } \; (x_1, \ldots, x_{n-1})\in {\XXX}^{n-1},\quad n\geq 2,
$$
where $K_n = D C^n (n!)^2$, and $D, C$ are some constants. Consequently,
\begin{equation}\label{34}
f_{i,j}(x_1, \ldots, x_n) \ \le \  K_{n-1} b(x_i, x_j),\quad (x_1, \ldots, x_{n})\in {\XXX}^{n}.
\end{equation}
Using (\ref{34}), the positivity of the operator $e^{t S_n}$ and estimate \eqref{positive} we have
\begin{align}
\begin{aligned}\label{36}
\big(e^{t S_n} f_{i,j} \big) (x_1, \ldots, x_{n})\ \le \ &  K_{n-1}\, \big(e^{t S_n}
b(\cdot_i, \cdot_j) \big)(x_1, \ldots, x_{n}) \\  \le \ &  K_{n-1} \, \big( e^{t \hat L_n^{\ast}}
b(\cdot_i, \cdot_j) \big)(x_1, \ldots, x_{n}).
\end{aligned}
\end{align}
Denote
\begin{align}
\begin{aligned}\label{21-bis}
{\mathcal{L}}^{i} k^{(n)}(x_1, \ldots, x_n) \ = &\int\limits_{\XXX} b(x_i, y) k^{(n)}(x_1, \ldots, x_{i-1}, y, x_{i+1}, \ldots,
x_n) \bar{\mathfrak{m}}( dy)\\
& - V(x_i) k^{(n)}(x_1, \ldots, x_n).
\end{aligned}
\end{align}
Using the critical regime condition \eqref{cr} we conclude that $e^{t {\mathcal{L}}^{i}}1\!\!1=1\!\!1$, $\forall i=1,\,\ldots, n,$ where
 $1\!\!1(x)\equiv 1$. Thus we get
\begin{equation}\label{36_1}
\left(e^{t \hat L_n^{\ast}}
b(\cdot_i, \cdot_j)\right) (x_1, \ldots, x_{n})\  =  \
\left(e^{t ( {\mathcal{L}}^i + {\mathcal{L}}^j)} b(\cdot_i, \cdot_j)\right) (x_1, \ldots, x_{n}).
\end{equation}
Note that the latter function depends only on variables $x_{i}$ and $x_{j}$.
\medskip

Notice that $e^{t \hat L_n^{\ast}} f_{i,j}$ is integrable with respect to $t$ on $\R_{+}$. Indeed, estimate \eqref{32}, condition \eqref{2b} and the identity
\begin{equation}\label{2BB}
 e^{t \hat L_n^{\ast}} b(x,y)   =   \mathbb{E}_{x,y} b(X(t), Y(t))
\end{equation}
imply that
\begin{equation}\label{2BB-bis}
\int_0^{\infty} e^{t \hat L_n^{\ast}} f_{i,j}(x_1, \ldots, x_{n}) \ dt \
\le  K_{n-1} \, H,
\end{equation}
where $H$ is the same constant as in \eqref{2b}.
Define
\begin{equation}\label{vij}
v^{(n)}_{i,j} \ = \ \int_0^{\infty} e^{t S_n} f_{i,j} \ dt,
\end{equation}
then    \eqref{36} and \eqref{2BB-bis} yield
\begin{equation}\label{vij-est}
v^{(n)}_{i,j} \ \le \ K_{n-1} \, H.
\end{equation}
\medskip

Starting from now, the proof of the main result completely repeats the reasoning given in the proof of Theorem 3.1 from \cite{PZ-EJP}. We briefly present next steps of the proof here for the reader's convenience.
Denote
\begin{equation}\label{k-def}
v^{(n)}  = \sum_{i \neq j} v^{(n)}_{i,j}=\int_0^{\infty} e^{t S_n} f^{(n)} dt, \qquad f^{(n)} = \sum_{i \neq j} f_{i,j},
\end{equation}
where $f_{i,j}$ was defined by \eqref{32}.
Next we prove that function $v^{(n)}$ is  a solution to \eqref{Last} in $\XX_{n}$.
It is easily seen from \eqref{vij} and induction procedure that $v^{(n)}\in\XX_{n}$. Since $e^{t S_n}$ is a strongly continuous semigroup we have
\begin{equation}\label{37A-bis}
e^{t S_n}f^{(n)} = f^{(n)} + S_n \int_{0}^{t}e^{s S_n}f^{(n)}ds.
\end{equation}
It was proved, see e.g. \cite{PZ-EJP}, that
the following limit holds in $\XX_{n}$:
\begin{equation}\label{2c-bis}
 e^{t \hat L_n^{\ast}}f^{(n)} \to 0, \quad t\to \infty.
\end{equation}
Consequently, using \eqref{positive} we conclude that
\begin{equation}\label{2c-bisS}
 e^{t S_n}f^{(n)} \to 0, \quad t\to \infty.
\end{equation}
A passage to the limit in \eqref{37A-bis}  as $t\to\infty$ together with \eqref{2c-bisS} shows  that $v^{(n)}$ defined in \eqref{k-def} can be taken as a solution $k^{(n)}$ to \eqref{Last}.

Since the function $f^{(n)}$ is the sum
of functions $f_{i,j}$, $i\neq j$ we deduce from \eqref{vij} that $ v^{(n)}$
is bounded by $n^2  K_{n-1} H$. Thus we get the recurrence inequality
\begin{equation}\label{50}
K_n \ \le \ n^2 K_{n-1} H,
\end{equation}
and by induction it follows that
\begin{equation}\label{49}
K_n \ \le \ H^n \, (n!)^2 \, k^{(1)}.
\end{equation}
Thus this solution $k^{(n)} = v^{(n)}$ with $v^{(n)}$ defined in \eqref{k-def} satisfies estimate
\begin{equation}\label{49A}
v^{(n)} (x_1, \ldots, x_n) \ \le \ H^n \, (n!)^2 \,  k^{(1)}.
\end{equation}

Moreover for \eqref{Phi} being a solution to the system
(\ref{Last}), any family of function of the form
$$
k_\varrho^{(1)} = \Phi_\varrho, \quad k^{(n)}_\varrho   =  \int\limits_0^{\infty} e^{t S_n} f^{(n)} \ dt  + \Upsilon^{(n)}_\varrho,\quad n \geq 2,
$$
is a solution to the system
(\ref{Last}), if  $\Upsilon^{(n)}_\varrho(x_1, \ldots, x_n)$ is an arbitrary function such that  $S_n \Upsilon^{(n)}_\varrho=0$. Taking
$$
\Upsilon^{(n)}_\varrho (x_1, \ldots, x_n) \ = \ \prod_{1}^{n} \Phi_\varrho(x_i)
$$
we conclude that
\begin{equation}\label{52}
\begin{array}{l}
k^{(1)}_\varrho = \Phi_\varrho,
\\[2mm]
k^{(n)}_\varrho = \int\limits_0^{\infty} e^{t  S_n} f^{(n)} dt  +  \prod\limits_{1}^{n} \Phi_\varrho(x_i) = v^{(n)}+\prod\limits_{1}^{n} \Phi_\varrho(x_i),\quad n\geq 2,
\end{array}
\end{equation}
is the desired solution to the stationary problem \eqref{Last} in the Banach spaces $(\XX_{n})_{n\geq 1}$.
To emphasize the dependence of $f^{(n)}$ on $\varrho$ (see \eqref{32}), we will use further notation $f_{\varrho}^{(n)}$  for $f^{(n)}$.

It follows from estimate \eqref{Phi-est} and formula \eqref{52} that for $\{ k_{\varrho}^{(n)} \}_{n\geq 1}$ the following recurrence inequality holds
\begin{equation}\label{53}
K_n \ \le \  n^2 K_{n-1} H \ + \ \varrho^n.
\end{equation}
Taking $L_n = \frac{K_n}{H^n (n!)^2}$ we get from \eqref{53}
$$
L_n \le L_{n-1}+ \frac{\varrho^n}{H^n (n!)^2} \le D \quad \forall \; n=1,2, \ldots; \qquad L_0=0.
$$
This yields
\begin{equation}\label{55}
K_n \ \le \ D H^n (n!)^2 \qquad \mbox{with} \; D = \sum\limits_{n=1}^{\infty} \frac{(\varrho / H)^n}{(n!)^2}.
\end{equation}
\medskip

To be certain that the constructed system $\{ k_{\varrho}^{(n)} \}_{n\geq 1}$ is a system of correlation functions, i.e. it corresponds to a probability measure $\mu^\varrho$ on the configuration space $\Gamma$, we will
prove below that $\{ k_{\varrho}^{(n)} \}_{n\geq 1}$ can be constructed as the limit when $t \to \infty$ of the system
of correlation functions $\{ k_{t}^{(n)}\}_{n\geq 1}$ associated with the solution to the Cauchy problem \eqref{59} with
the initial data \eqref{k0}.


We recall that by the variation of parameters formula we have relation \eqref{61A} for the solution to the Cauchy problem (\ref{59}).
On the other hand, we proved above the existence of the solution $\{ k_{\varrho}^{(n)} \}_{n\geq 1}$ of the stationary problem:
\begin{equation}\label{New1}
S_n  k_\varrho^{(n)} \ = \ - f_\varrho^{(n)},
\end{equation}
with
$$
f_\varrho^{(n)}(x_1, \ldots, x_n) \ = \ \sum_{i,j:\ i\neq j}
k_\varrho^{(n-1)}(x_1, \ldots,\check{x_i}, \ldots, x_n) \ b(x_i, x_j).
$$
This solution is given by formula \eqref{52}, and
\eqref{New1} implies the following relation
\begin{equation}\label{New2}
\left( e^{t S_n} -  1 \right) k_\varrho^{(n)} \ = \ -
\int\limits_0^t \frac{d}{ds} e^{(t-s) S_n} k_\varrho^{(n)} ds \
\ = \ - \int\limits_0^t e^{(t-s) S_n}  f_\varrho^{(n)} \ ds.
\end{equation}
Therefore from \eqref{61A} and \eqref{New2} we obtain
\begin{equation}\label{64}
k_{t}^{(n)} -  k_\varrho^{(n)} \ = \
e^{t S_n}(k_{0}^{(n)} - k_\varrho^{(n)}) \ + \  \int\limits_0^t
e^{(t-s) S_n} (f_{s}^{(n)} -  f_\varrho^{(n)}) \ ds.
\end{equation}
We will prove now that both terms in the right-hand side of (\ref{64}) converge to 0 in
the norm of $\XX_n$ as $t \to \infty$.

Formula (\ref{52}) and equality $S_n \prod\limits_{1}^{n} \Phi_\varrho(x_i) = 0$ yield
\begin{equation}\label{65}
e^{t S_n} \big( k_{0}^{(n)} - k_\varrho^{(n)} \big) \ = \ - e^{t S_n} v^{(n)} + \Big( e^{t S_n} k_{0}^{(n)} - \prod\limits_{1}^{n} \Phi_\varrho(x_i) \Big).
\end{equation}
The second term in \eqref{65} tends to 0 as $t \to \infty$ because by \eqref{Phi} we get
$$
e^{t S_1} k_{0}^{(1)}(x) = e^{t S_1} \varrho  \to \Phi_\varrho(x),
$$
and $ k_{0}^{(n)} (x_1, \ldots, x_n) = \prod\limits_{1}^{n} k_{0}^{(1)}(x_i) = \varrho^n$ (see \eqref{k0}).
According \eqref{k-def} the first term in the r.h.s. of  \eqref{65} can be rewritten as
\begin{equation*}\label{1termA}
e^{t S_n} \ v^{(n)} = \int_{0}^{\infty}e^{(t+s) S_n}f_\varrho^{(n)} \, ds =  \int_{t}^{\infty}e^{r S_n} f_\varrho^{(n)} \, dr  \le  \int_{t}^{\infty} e^{r  \hat L_n^{\ast}} f_\varrho^{(n)} \, dr .
\end{equation*}
Due to the structure \eqref{32} of the function $f_\varrho^{(n)}$ and the uniform convergence of the integral in \eqref{2b} we conclude that
\begin{equation}\label{1term}
||e^{t S_n} \ v^{(n)}||_{{\XX}_n}\to 0,\quad t\to\infty.
\end{equation}
The second term in the r.h.s. of   \eqref{64} also tends to 0, and it can be proved in the same way as in our previous works, see e.g. \cite{KPZh}.

Thus we proved the strong convergence (\ref{Th1-2}), and the proof of the second part of Theorem \ref{main} is completed.
\medskip

Now we go back to the first part of the Theorem \ref{main}, and the final step of the proof is to show that the system of correlation functions $\{k_\varrho^{(n)} \}_{n \ge 1}$ corresponds to a probability measure $\mu^\varrho$ on the configuration space $\Gamma$. For this  we have constructed above $k_\varrho^{(n)}$ as the limit when $t \to \infty$ of solution $ k_{t}^{(n)} $
of the Cauchy problem (\ref{59}) with initial data (\ref{k0}):
\begin{equation}\label{limk}
k^{(n)}_\varrho \ = \ \lim_{t\to\infty} k_{t}^{(n)}.
\end{equation}
Then one can prove that solution $ k_{t}^{(n)}$
of the Cauchy problem satisfies the Lenard positivity and the moment growth conditions, see \cite{L1}-\cite{L2}.
The detailed proof of this fact can be found in \cite{KKPZ}.
Finally, these conditions imply that for any  $\varrho >0$  there exists a unique probability measure $\mu^\varrho \in {\cal M}_{\rm{corr}}(\Gamma)$, whose correlation functions are $\{ k^{(n)}_\varrho\}_{n \ge 1}$.
This completed the proof of Theorem \ref{main}.

\end{document}